\documentclass{amsart}
\usepackage{amssymb}
\theoremstyle{plain}

\newcommand{\Iso}{\mathrm{Iso}}
\newcommand{\w}{\omega}
\newcommand{\U}{\mathcal U}
\newcommand{\e}{\varepsilon}

\newcommand{\II}{\mathbb I}
\newcommand{\dens}{\mathrm{dens}}
\newcommand{\dist}{\mathrm{dist}}
\newcommand{\diam}{\mathrm{diam}}
\newcommand{\IR}{\mathbb R}
\newcommand{\LC[1]}{\mathrm{LC}^{#1}}
\newcommand{\HH}{\mathcal H}

\newtheorem{theorem}{Theorem}[section]
\newtheorem{corollary}[theorem]{Corollary}
\newtheorem{problem}[theorem]{Problem}
\newtheorem{lemma}[theorem]{Lemma}
\newtheorem{claim}[theorem]{Claim}

\theoremstyle{definition}
\newtheorem{example}[theorem]{Example}

\begin{document}

\title[Detecting Hilbert manifolds among homogeneous  spaces]{Detecting Hilbert manifolds among\\ isometrically homogeneous metric spaces}

\author[T.O.~Banakh and D.~ Repov\v s]{Taras O. Banakh and Du\v san Repov\v s}

\address{Department of Mathematics, Ivan Franko National University of Lviv, and\newline
Instytut Matematyki, Uniwersytet Humanistyczno-Przyrodniczy Jana Kochanowskiego w Kielcach, Poland}
\email{tbanakh@yahoo.com}

\address{Faculty of Mathematics and Physics, and
Faculty of Education,
University of Ljubljana,
P. O. Box 2964,
Ljubljana, Slovenia 1001}
\email{dusan.repovs@guest.arnes.si}
 
\subjclass[2010]{57N16; 57N20; 54E35}
\keywords{Hilbert manifold, ANR, Locally Finite Approximation Property,
isometrically homogeneous metric space, topological group, balanced subgroup}

\begin{abstract} 
We detect Hilbert manifolds among isometrically homogeneous metric spaces 
and apply the obtained results to recognizing Hilbert manifolds 
among homogeneous spaces of the form $G/H$,
where $G$ is a metrizable topological group
and $H$ is a closed balanced subgroup of $G$.
\end{abstract}
\date{\today}
\maketitle

\section{Introduction}

The problem of detecting topological groups that are locally homeomorphic 
to (finite or infinite)-dimensional
Hilbert spaces traces its history back
to the fifth problem of David Hilbert 
concerning
the recognition of Lie groups in the class of topological groups. 
This problem was resolved by combined efforts of 
A.~Gleason \cite{Gle}, 
D.~Montgomery, L.~Zippin \cite{MZ}, 
and K.~Hofmann \cite{Hofmann}. 
According to their results, a topological group $G$ is a Lie group 
if and only if $G$ is locally compact and locally contractible. 
In this case $G$ is an Euclidean manifold, that is, 
a manifold modeled on an Euclidean space $\IR^n$.

The next step was made in 1981 by T.~Dobrowolski and H.~Toru\'nczyk \cite{DT}. 
They proved that a topological group $G$ is a manifold modeled on a separable
Hilbert space if and only if $G$ is a locally Polish ANR. 
A topological space is called {\em locally Polish} if each
point $x\in X$ has a Polish (i.e. separable completely metrizable) neighborhood.

Most recently, T.~Banakh and I.~Zarichnyy \cite{BZ} proved in 2008 
that a topological group $G$ is a manifold modeled on an 
infinite-dimensional Hilbert space if and only if $G$ is a completely metrizable ANR with LFAP.
A topological space $X$ is said to have  
{\em Locally Finite Approximation Property} (abbreviated  LFAP) 
if for every open cover $\U$ there are maps $f_n:X\to X$, $n\in\w$, 
such that each $f_n$ is $\U$-near to the identity map and the 
indexed family $\{f_n(X)\}_{n\in\w}$ is locally finite in $X$. 
This property was crucial in Toru\'nczyk's characterization \cite{Tor81} 
of non-separable Hilbert manifolds. 
\medskip

By the Birkhoff-Kakutani Metrization Theorem \cite[2.5]{Tk}, 
the topology of any first countable topological group $G$ is 
generated by a left-invariant metric. This metric turns $G$ 
into an isometrically homogeneous metric space.
We define a metric space $X$ to be {\em isometrically homogeneous} if for any 
two points $x,y\in X$ there is a bijective isometry $f:X\to X$ 
such that $f(x)=y$. This notion is a metric analogue
of the well-known notion of a topologically homogeneous spaces.
We recall that a topological space $X$ is called {\em topologically 
homogeneous} if for any two points $x,y\in X$ there is a homeomorphism
$f:X\to X$ such that $f(x)=y$.

In light of the mentioned results the following open problem arises naturally:

\begin{problem} \label{prob1} 
How can one detect Euclidean and Hilbert manifolds among isometrically homogeneous metric spaces?
\end{problem}

For the Euclidean case of this problem we have the following answer 
which will be derived in Section~\ref{tbg}
from a result of J.~Szenthe \cite{Szenthe}. 

\begin{theorem}\label{t1} 
An isometrically homogeneous metric space $X$ is an Euclidean manifold if and only if 
$X$ is locally compact and locally contractible.
\end{theorem}

The Hilbert case
of Problem~\ref{prob1} is more difficult. We shall answer this problem under 
an addition assumption that the isometrically homogeneous space is $\II^{<\w}{\sim}$homogeneous.
The class of such spaces includes all metric groups 
(that is, topological groups endowed with an admissible  left-invariant metric) 
and also quotient spaces $G/H$ of metric groups $G$ by closed balanced subgroups
$H\subset G$ (cf. Corollary~\ref{cor1}).

To introduce $\II^{<\w}{\sim}$homogeneous metric spaces, let us first observe that a metric space $X$ is isometrically homogeneous if and only if the action of the isometry group $\Iso(X)$ on $X$ is transitive. This is equivalent to saying that for each point $\theta\in X$ the map $$\alpha_\theta:\Iso(X)\to X,\;\;\alpha_\theta:f\mapsto f(\theta),$$
is surjective. 

It is well-known (and easy to check) that the isometry group $\Iso(X)$ of a metric space $X$ is a topological group with respect to the topology of pointwise convergence (that is, the topology inherited from the Tychonov power $X^X$). Moreover, the natural action 
$$\alpha:\Iso(X)\times X\to X,\;\alpha:(f,x)\mapsto f(x),$$
of $\Iso(X)$ on $X$ is continuous.

Let $T$ be a topological space. We  define a map $q:X\to Y$ between topological spaces to be 
\begin{itemize}
\item {\em $T{-}$invertible} if for each continuous map $f:T\to Y$ there is a continuous map $g:T\to X$ such that $q\circ g=f$;
\item {\em $T{\sim}$invertible} if for each continuous map $f:T\to Y$ and an open cover $\U$ of $Y$ there is a continuous map $g:T\to X$ such that $q\circ g$ is $\U$-near to $f$ (in the sense that for each $t\in T$ there is $U\in\U$ with $\{f(t),q\circ g(t)\}\subset U$). 
\end{itemize}

Observe that a map $q:X\to Y$ is $\II^0{-}$invertible if and only if $q(X)=Y$ and $q$ is $\II^0{\sim}$invertible if and only if $q(X)$ is dense in $Y$ (here $\II^0$ is a singleton). 

We define a metric space $X$ to be {\em $T{-}$homogeneous} (resp. {\em $T{\sim}$homogeneous}), where $T$ is a topological space, if for some point $\theta\in X$ the map $\alpha_\theta:\Iso(X)\to X$ is $T$-invertible (resp. $T{\sim}$invertible).

Let us observe that each metric group $G$ (that is, a topological group endowed with an admissible  left-invariant metric) is a $G$-homogeneous metric space. This follows from the 
fact that for the neutral element $\theta$ of $G$ the map $\alpha_\theta:\Iso(G)\to G$ admits a continuous section $l:G\to \Iso(G)$ defined by $s:g\mapsto l_g$, where $l_g:x\mapsto gx$, is the left shift.

We shall be interested in the $T{-}$ and $T{\sim}$homogeneity in case $T$ is a (finite- or infinite-dimensional) cube $\II^n$. 
Observe that a metric space $X$ is $\II^0{-}$homogeneous if and only if $X$ is isometrically homogeneous, and $X$ is $\II^0{\sim}$homogeneous if and only if some point $\theta\in X$ has dense orbit under the action of the isometry group $\Iso(X)$.

On the other hand, a metric space $X$ is $\II^n{-}$homogeneous for all $n\in\w$ if and only if $X$ is $\II^{<\w}{-}$homogeneous for the topological sum $\II^{<\w}=\oplus_{n\in\w}\II^n$ of finite-dimensional cubes. A metric space $X$ is $\II^{<\w}{\sim}$homogeneous if and only if it is $\II^\w{\sim}$homogeneous. 
%One can also observe that $X$ is $\II^{<\w}{-}$homogeneous if and only if  for every $\theta\in X$ the map $\alpha_\theta:\Iso(X)\to X$ is a Serre fibration (which means that $\alpha_\theta$ has the property of lifting homotopies over CW-complexes).
\smallskip

For each metric space $X$ those homogeneity properties relate as follows:
\smallskip

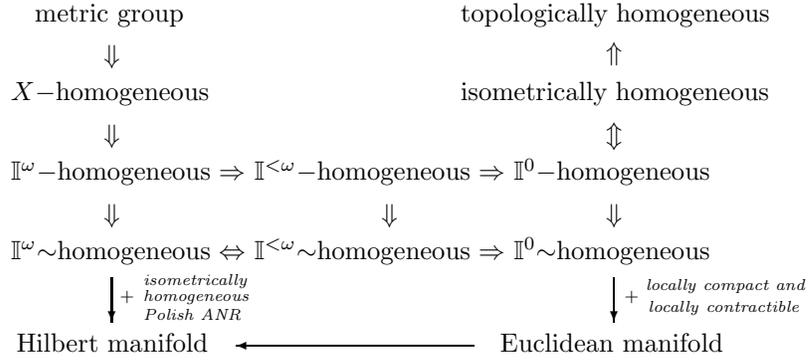
{%\small
\begin{picture}(300,160)(-30,-28)

\put(9,110){metric group}

\put(170,110){topologically homogeneous}
\put(225,95){$\Uparrow$}
\put(35,94){$\Downarrow$}
\put(0,80){$X{-}$homogeneous}
\put(35,64){$\Downarrow$}
\put(170,80){isometrically homogeneous}
\put(225,64){$\Updownarrow$}

\put(0,50){$\II^\w{-}$homogeneous $\Rightarrow$ $\II^{<\w}{-}$homogeneous $\Rightarrow$ $\II^{0}{-}$homogeneous}
\put(35,34){$\Downarrow$}

\put(0,20){$\II^\w{\sim}$homogeneous $\Leftrightarrow$ $\II^{<\w}{\sim}$homogeneous $\Rightarrow$ $\II^{0}{\sim}$homogeneous}

\put(140,34){$\Downarrow$}
\put(225,34){$\Downarrow$}

\put(175,-13){\vector(-1,0){90}}
\put(38,13){\vector(0,-1){17}}
\put(41,4){\tiny +}
\put(50,10){\tiny{\it isometrically}}
\put(50,4){\tiny{\it homogeneous}} 
\put(50,-3){\tiny{\it Polish ANR}}
\put(228,13){\vector(0,-1){17}}

\put(240,8){\tiny{\it locally compact and}} 
\put(232,4){\tiny +}
\put(241,0){\tiny{\it  locally contractible}}
\put(2,-15){Hilbert manifold}
\put(185,-15){Euclidean manifold}
\end{picture}
}

The last two implications in the diagram hold under additional assumptions on the local structure of $X$ and are established in the following theorem that recognizes Hilbert manifolds among $\II^{<\w}{\sim}$homogeneous metric spaces (and will be proved in Section~\ref{pf-main}). 

\begin{theorem}\label{main} An isometrically homogeneous $\II^{<\w}{\sim}$homogeneous metric space $X$ is a manifold modeled on 
\begin{enumerate} 
\item an Euclidean space if and only if $X$ is locally precompact, locally Polish, and locally contractible;
\item a separable Hilbert space if and only if $X$ is a locally Polish ANR;
\item an infinite-dimensional Hilbert space if and only if $X$ is  completely-metrizable ANR with LFAP.
\end{enumerate}
\end{theorem}

We explain some of the notions appearing in this theorem. 
A metric space is said to be
({\em locally}) {\em precompact} if its completion is (locally) compact.
A topological space $X$ is called
{\em locally Polish} if each point of $X$ has a Polish (= separable completely-metrizable) neighborhood; $X$ is said to be 
{\em completely-metrizable} if its topology is generated by a complete metric.
ANR is the standard
abbreviation for the absolute neighborhood retracts
in the class of metrizable spaces. 
\medskip

\section{Detecting Hilbert manifolds among quotient spaces of topological groups}

In this section we shall apply Theorem~\ref{main} to detecting Hilbert manifolds among homogeneous spaces of the form $G/H=\{xH:x\in G\}$ where $H$ is a closed subgroup of a topological group $G$ and $G/H$ is endowed with the quotient topology. 
We define a subgroup $H$ of a topological group $G$ to be {\em balanced} if for every neighborhood $U\subset G$ of the neutral element $e\in G$ there is a neighborhood $V\subset G$ of $e$ such that $HV\subset UH$. 

\begin{corollary}\label{cor1} Let $H\subset G$ be a balanced closed subgroup of a metrizable topological group $G$ such that the quotient map $q:G\to G/H$ is  $\II^{<\w}{\sim}$invertible. The space $G/H$ is a manifold modeled on 
\begin{enumerate}
\item an Euclidean space if and only if $G/H$ is locally compact and locally contractible;
\item a separable Hilbert space if and only if $G/H$ is a locally Polish ANR;
\item an infinite-dimensional Hilbert space if and only if $G/H$ is a completely-metrizable ANR with LFAP.
\end{enumerate}
\end{corollary}

\begin{proof} By the Birkhoff-Kakutani Theorem \cite[2.5]{Tk}, the topology of $G$ is generated by a bounded left-invariant metric $d$. This metric induces the Hausdorff metric
$$d_H(A,B)=\max\{\sup_{a\in A}d(a,B),\sup_{b\in B}d(b,A)\}$$on the hyperspace $2^G$ of all non-empty closed subsets of $G$. 

Endow the quotient space $G/H=\{xH:x\in G\}$ with the Hausdorff metric $d_H$ and observe that for each $g\in G$ the left shift $l_g:G/H\to G/H$, $l_g:xH\mapsto gxH$, is an isometry of $G/H$. Therefore, the Hausdorff metric turns $G/H$ into an isometrically homogeneous metric space. 

We claim that this metric generates the quotient topology on $G/H$. Because of the homogeneity, it suffices to check that $d_H$ generates the quotient topology at the distinguished element $H$ of $G/H$. 

Fix a basic neighborhood $U\cdot H=\{uH:u\in U\}\subset G/H$ of $H$, where $U\subset G$ is a neighborhood of the neutral element $e$ in $G$. Since $H$ is balanced, there is a neighborhood $V\subset G$ of $e$ such that $HV\subset UH$. Find $\e>0$ such that 
$B_\e\subset V$ where $B_\e=\{x\in G:d(x,e)<\e\}$ is the $\e$-ball centered at $e$. Then for each coset $xH\in G/H$ with $d_H(xH,H)<\e$ we get $xH\subset HV\subset UH$. This shows that the topology on $G/H$ generated by the Hausdorff metric $d_H$ is stronger than the quotient topology.

Next, given any $\e>0$, use the balanced property of $H$ to find a neighborhood $V=V^{-1}\subset G$ of $e$ such that $HV\subset B_\e H$.
Then $VH=(HV)^{-1}\subset (B_\e H)^{-1}\subset HB_\e$. Consequently, for every $v\in V$ we get $vH\subset HB_{\e}$. Since $v^{-1}\in V$ we also get $v^{-1}H\subset HB_{\e}$ and $H\subset vHB_\e$. The inclusions $vH\in HB_{\e}$ and $H\subset vHB_\e$ imply that $d_H(H,vH)\le\e$. Consequently, $V\cdot H\subset\{g\in G:d_H(gH,H)\le\e<2\e\}$, which shows that the quotient topology on $G/H$ is stronger than the topology generated by the Hausdorff metric $d_H$ on $G/H$.
\smallskip

The $\II^{<\w}{\sim}$invertibility of the quotient map $q:G\to G/H$ implies the $\II^{<\w}{\sim}$homo\-geneity of the isometrically homogeneous metric space $(G/H,d_H)$. Now the statements (1)--(3) follow immediately from Theorem~\ref{main}.
\end{proof}

A topological space $X$ is defined to be $\LC[<\w]$ if for each point $x\in X$, each neighborhood $U\subset X$ of $x$, and every $k<\w$ there is a neighborhood $V\subset U$ of $x$ such that each map $f:S^k\to V$ is null homotopic in $U$. 

\begin{corollary}\label{cor2} Let $H\subset G$ be a completely-metrizable balanced $LC^{<\w}$-subgroup of a metrizable topological group $G$. The space $G/H$ is a manifold modeled on 
\begin{enumerate}
\item an Euclidean space if and only if $G/H$ is locally compact and locally contractible;
\item a separable Hilbert space if and only if $G/H$ is a locally Polish ANR;
\item an infinite-dimensional Hilbert space if and only if $G/H$ is a completely-metrizable ANR with LFAP.
\end{enumerate}
\end{corollary}

\begin{proof} This corollary will follow from Corollary~\ref{cor1} as soon as we check that the quotient map $q:G\to G/H$ is $\II^{<\w}{-}$invertible. For this we shall apply the Finite-Dimensional Selection Theorem of E.Michael \cite{Mi}.

Let $d$ be a left-invariant metric generating the topology of the group $G$. This metric induces an admissible metric $\rho(x,y)=d(x,y)+d(x^{-1},y^{-1})$ on $G$. It is well-known that the completion $\bar G$ of $G$ by the metric $\rho$ has the structure of topological group. The subgroup $H\subset G\subset\bar G$, being completely-metrizable, is closed in $\bar G$. 

The $\II^{<\w}{-}$invertibility of the quotient map $q:G\to G/H$ will follow from the Michael Selection Theorem \cite{Mi} as soon as we check that the family $\{xH:x\in G\}$ is equi-$\LC[n]$ for every $n\in\w$. The latter means that for every $x_0\in G$ and a neighborhood $U(x_0)\subset G$ of $x_0$ there is another neighborhood $V(x_0)\subset U(x_0)$ of $x_0$ such that each map $f:S^n\to xH\cap V(x_0)$ from the $n$-dimensional sphere into a coset $xH\in G/H$, $x\in G$, is null homotopic in $xH\cap U(x_0)$.

Find a neighborhood $U\subset G$ of the neutral element $e$ of $G$ such that $x_0U^2\subset U(x_0)$. Since $H$ is $\LC[<\w]$, there is a neighborhood $W\subset G$ of $e$ such that each map $f:S^n\to H\cap W$ is null homotopic in $U\cap H$.
Find a neighborhood $V\subset U$ of $e$ such that $x_0^{-1}V^{-1}Vx_0\subset W$.

We claim that the neighborhood $V(x_0)=Vx_0\cap x_0V$ has the desired property. Indeed, fix any map $f:S^n\to xH\cap V(x_0)$ where $x\in V(x_0)$. Consider the left shift $l_{x^{-1}}:g\mapsto x^{-1}g$, and observe that $$l_{x^{-1}}\circ f(S^n)\subset H\cap x^{-1}V(x_0)\subset x_0^{-1}V^{-1}Vx_0\subset W.$$ Now the choice of $W$ ensures that the map $l_{x^{-1}}\circ f$ is null-homotopic in $H\cap U$ and hence $f$ is null-homotopic in $$xH\cap xU\subset xH\cap x_0VU\subset xH\cap x_0U^2\subset xH\cap U(x_0).$$
\end{proof}
 
Since the trivial subgroup is balanced, Corollary~\ref{cor2} implies the following three results due to K.Hofmann \cite{Hofmann}, T.Dobrowolski, H.Toru\'nczyk \cite{DT}, and T.Banakh, I.Zarichnyy \cite{BZ}, respectively. 

\begin{corollary}\label{cor3} A topological group $G$ is a manifold modeled on 
\begin{enumerate}
\item an Euclidean space if and only if $G$ is locally compact and locally contractible;
\item a separable Hilbert space if and only if $G$ is a locally Polish ANR;
\item an infinite-dimensional Hilbert space if and only if $G$ is a completely-metrizable ANR with LFAP.
\end{enumerate}
\end{corollary}

In should be mentioned that the requirement on the subgroup $H\subset G$ to be balanced is essential in Corollaries~\ref{cor1} and \ref{cor2}.

\begin{example} By \cite{Fer}, the homeomorphism group $\HH(\II^\w)$ of the Hilbert cube $\II^\w$ is a Polish ANR. Moreover, by Corollary 4.12 of \cite{Fer}, for any point $\theta\in \II^\w$ the closed subgroup $\HH_\theta(\II^\w)=\{h\in\HH(\II^\w):h(\theta)=\theta\}\subset\HH(\II^\w)$ is an ANR as well. This subgroup is not balanced because otherwise the Hilbert cube  $\II^\w=\HH(\II^\w)/\HH_\theta(\II^\w)$ would be an Euclidean manifold by Corollary~\ref{cor2}(1).

Next, we show that the quotient map $q:\HH(\II^\w)\to \HH(\II^\w)/\HH_\theta(\II^\w)$ is $\II^\w{\sim}$invertible but not $\II^\w{-}$invertible. The $\II^\w{\sim}$invertibility of  $q$ follows from the $\LC[<\w]$-property of the subgroup $\HH_\theta(\II^\w)$ and Finite-Dimensional Michael Selection Theorem \cite{Mi}. On the other hand, the fixed point property of $\II^\w$ implies that the quotient map $q$ is not $\II^\w$-invertible. Indeed, assuming that $q$ has a section $s:\II^\w\to \HH(\II^\w)$, $s:x\mapsto s_x\in \HH(\II^\w)$, and taking any homeomorphism $g\in \HH(\II^\w)\setminus \HH_\theta(\II^\w)$, we would get a continuous map $f:\II^\w\to \II^\w$, $f:x\mapsto q(s_x\circ g)$, without fixed point. Indeed, assuming that $f(x)=x$ for some $x\in\II^\w$, we would get $$x=f(x)=q(s_x\circ g)=s_x\circ g(\theta).$$ Since $x=q(s_x)=s_x(\theta)$, this would imply that $g(\theta)=\theta$ and hence $g\in\HH_\theta(\II^\w)$, which contradicts the choice of the homeomorphism $g$.
\end{example}

\begin{problem} Let $H$ be a closed subgroup of a Polish ANR-group $G$ such that the quotient map $q:G\to G/H$ is a locally trivial bundle. Is the quotient space $G/H$ a Hilbert manifold?
\end{problem}

Another related problem was posed in \cite{HR}:

\begin{problem} Let $H$ be a closed ANR-subgroup of a Polish ANR group $G$. Is $G/H$ a manifold modeled on a Hilbert space or the Hilbert cube?
\end{problem}

\section{Locally precompact isometrically homogeneous metric spaces}\label{tbg}

In this section we shall study locally precompact isometrically homogeneous metric spaces. We recall that  a metric space is locally precompact if its completion is locally compact. The following theorem implies Theorems~\ref{t1} announced in the introduction.

\begin{theorem}\label{szenthe} An isometrically homogeneous metric space $X$ is an Euclidean manifold if and only if $X$ is locally precompact, locally Polish, and locally contractible.
\end{theorem}

\begin{proof} If  $X$ is an Euclidean manifold, then $X$ is locally compact and locally contractible. The local precompactness of $X$ will follow as soon as we show that $X$ is complete.
Take any point $x_0$ in the completion $\bar X$ of the metric space $X$.  Fix any point $\theta\in X$ and by the local compactness of $X$, find an $\e>0$ such the closed $\e$-ball $B(\theta,\e)=\{x\in X:\dist(x,\theta)\le\e\}$ is compact. Since $X$ is isometrically homogeneous, there is a homeomorphism $f:X\to X$ such that $\dist(x_0,f(\theta))<\e/2$. It follows that the $\e$-ball $B(f(\theta),\e)=\{x\in X:\dist(x,f(\theta))\le\e\}$ contains the point $x_0$ in its closure in $\bar X$. Since $B(f(\theta),\e)=f(B(\theta,\e))$ is compact, $x_0\in B(f(\theta),\e)\subset X$. 
Thus $X=\bar X$ is a complete metric space. Being locally compact, this space is locally precompact.
\smallskip

Now assume that $X$ is locally precompact, locally Polish, and locally contractible. We need to show that $X$ is an Euclidean manifold. It suffices to check that each connected component of $X$ is an Euclidean manifold. Since connected components of $X$ are isometrically homogeneous, we lose no generality assuming that $X$ is connected.

The completion $\bar X$ of the locally precompact space $X$ is locally compact.
By \cite{DW} (cf. also \cite[Th.I.4.7]{KN}), the isometry group $\Iso(\bar X)$ of $\bar X$ is locally compact, metrizable, and separable. Moreover, for every point $\theta\in X$ the action $$\alpha_\theta:\Iso(\bar X)\to \bar X$$ is proper in the sense that it is closed and the stabilizer $\Iso(\bar X,\theta)=\{f\in\Iso(\bar X):f(\theta)=\theta\}$ is compact (cf.  \cite{MS}). 

We are going to prove that the metric space $X$ is complete. For this consider the subgroup
$$\Iso(X)=\{f\in\Iso(\bar X):f(X)=X\}\subset\Iso(\bar X)$$ 
in the group $\Iso(\bar X)$. 

Let us show that the subgroup $\Iso(X)=\{f\in\Iso(\bar X):f(X)=X\}$ of $\Iso(\bar X)$ 
is coanalytic. The subspace $X$ of $\bar X$, being (locally) Polish, is a $G_\delta$-set in $\bar X$. Consequently, its complement $\bar X\setminus X$ can be written as the countable union $\bar X\setminus X=\bigcup_{n\in\w}K_n$ of non-empty compact sets. Observe that $$\Iso(X)=\bigcap_{n\in\w}\{f\in\Iso(\bar X):f(K_n)\cup f^{-1}(K_n)\subset \bar X\setminus X\}.$$ Let $\exp(\bar X)$ be the space of non-empty compact subset of $\bar X$ endowed with the Hausdorff metric. For every $n\in\w$ consider the continuous maps 
$$\xi_n:\Iso(\bar X)\to\exp(\bar X),\;\;\;\xi_n:f\mapsto f(K_n)\cup f^{-1}(K_n)
.$$

By \cite[33.B]{Ke}, the subspace $\exp(\bar X\setminus X)=\{K\in\exp(\bar X):K\subset\bar X\setminus X\}$ is coanalytic and so is its preimage $\xi_n^{-1}(\exp(\bar X\setminus X))$ for $n\in\w$. Since $$\Iso(X)=\bigcap_{n\in\w}\xi_n^{-1}(\exp(\bar X\setminus X)),$$ we see that the subgroup $\Iso(X)$ is coanalytic in $\Iso(\bar X)$. Then $\Iso(X)$ has the Baire property in $G$ and hence either is meager or is closed in $\Iso(\bar X)$ according to \cite[9.9]{Ke}.

If $\Iso(X)$ is closed in $\Iso(\bar X)$, then $X=\alpha_\theta(\Iso(X))=\bar X$ (because the map $\alpha_\theta:\Iso(\bar X)\to\bar X$ is closed) and we are done.
It remains to prove that the assumption that $\Iso(X)$ is meager leads to a contradiction. Let $G$ be the closure of the subgroup $\Iso(X)$ in $\Iso(\bar X)$.

Taking into account that the map $\alpha_\theta:\Iso(\bar X)\to \bar X$ is closed and $X=\alpha_\theta(\Iso(X))$ is dense in $\bar X$, we conclude that $\bar X =\alpha_\theta(G)$. It follows from the local compactness of $G$ and the properness of the action $\alpha_\theta:G\to\bar X$ that the map $\alpha_\theta$ is open. Then the image $\alpha_\theta(\Iso(X))=X$ of the meager subgroup $\Iso(X)$ of $G$ is a meager subset of $\bar X$, which is not possible as $X$ is a dense $G_\delta$-set in $\bar X$.
This contradiction completes the proof of the completeness of $X$.

Now we see that the connected locally compact locally contractible space $X=\bar X$ admits an effective transitive action of the locally compact group $\Iso(X)$. By Theorem~3 of J.~Szenthe \cite{Szenthe}, $\Iso(X)$ is a Lie group and $X$ is an Euclidean manifold.
\end{proof}

\section{Characterizing the topology of Hilbert manifolds}

In order to prove Theorem~\ref{main}(2,3) we shall apply  the celebrated
Toru\'nczyk's characterization of the topology of infinite-dimensional Hilbert manifolds. The key ingredient of this characterization is the $\kappa$-discrete $m$-cells property defined for cardinals $\kappa$ and $m$ as follows.

We say that a topological space $X$ satisfies the {\em $\kappa$-discrete $m$-cells property} if for every map $f:\kappa\times \II^m\to X$ and every open cover $\U$ of $X$ there is a map $g:\kappa\times \II^m\to X$ such that $g$ is $\U$-near to $f$ and the family $\big\{g(\{\alpha\}\times \II^m)\big\}_{\alpha\in\kappa}$ is discrete in $X$ (here we identify the cardinal $\kappa$ with the discrete space of all ordinals $<\kappa$).

The following characterization theorem is due to H.Toru\'nczyk \cite{Tor81}.

\begin{theorem}[Toru\'nczyk]\label{tor1} A metrizable space $X$ is a manifold modeled on an infinite-dimensional Hilbert space $l_2(\kappa)$ of density $\kappa\ge\w$ if and only if $X$ has the following properties:
\begin{enumerate}
\item $X$ is a completely metrizable ANR;
\item each connected component of $X$ has density $\le\kappa$;
\item $X$ has the $\kappa$-discrete $m$-cells property for all $m<\w$; and
\item $X$ has LFAP.
\end{enumerate}
\end{theorem} 

For manifolds modeled on the separable Hilbert space $l_2$ this characterization can be simplified as follows:

\begin{theorem}[Toru\'nczyk]\label{tor2} A metrizable space $X$ is an $l_2$-manifold if and only if $X$ is a locally Polish ANR with the $\w$-discrete $\w$-cells property.
\end{theorem}

Thus the problem of recognition of Hilbert manifolds reduces to detecting the $\kappa$-discrete $m$-cells property. For spaces with LFAP the latter problem can be reduced to cardinals $\kappa$ with uncountable cofinality. The following lemma is proved in \cite{BZ}.

\begin{lemma}\label{l1} A paracompact space 
$X$ with $\w$-LFAP has $\kappa$-discrete $m$-cells property for a cardinal $\kappa$  if and only if  $X$ has the $\lambda$-discrete $m$-cells property for all cardinals $\lambda\le\kappa$ of uncountable cofinality.
\end{lemma}

In fact, the $\kappa$-discrete $m$-cells property follows from its metric counterpart called the $\kappa$-separated $m$-cells property.

Following \cite{BZ}, we define a metric space $(X,\rho)$ to have the {\em $\kappa$-separated $m$-cells property} if for every $\e>0$ there is $\delta>0$ such that for every map $f:\kappa\times \II^m\to X$ there is a map $g:\kappa\times \II^m\to X$ that is $\e$-homotopic to $f$ and such that $$\dist\big(g(\{\alpha\}\times \II^m),g(\{\beta\}\times \II^m)\big)\ge \delta$$ for all ordinals $\alpha<\beta<\kappa$.

The following lemma was proved in \cite{BZ} by the method of the proof of Lemma 1 in \cite{DT}.

\begin{lemma}\label{l2} Each metric space $X$ with the $\kappa$-separated $m$-cells property has the $\kappa$-discrete $m$-cells property.
\end{lemma}

According to Lemma~6 of \cite{BZ}, the $\kappa$-separated $m$-cells property can be characterized as follows:

\begin{lemma}\label{l3} Let $m\le \w\le\kappa$ be two cardinals. A metric space $X$ has the $\kappa$-separated $m$-cells property if and only if for every $\e>0$ there is $\delta>0$ such that for every subset $A\subset X$ of cardinality $|A|<\kappa$, and every map $f:\II^d\to X$ of a cube of finite dimension $d\le m$ there is a map $g:\II^d\to X$ that is $\e$-homotopic to $f$ and has $\dist(g(\II^d),A)\ge\delta$.
\end{lemma}

\section{The $\kappa$-separated $m$-cells property in metric spaces}

In this section we shall establish the $\kappa$-separated $m$-cells property in $\II^m{\sim}$homo\-geneous metric spaces. A subset $S$ of a metric space $X$ is called {\em separated} if it is {\em $\e$-separated} for some $\e>0$. The latter means that $\dist(x,y)\ge\e$ for any distinct points $x,y\in S$. 

\begin{lemma}\label{l4} Let $m\le \w\le\kappa$ be two cardinals. An $\II^m{\sim}$homogeneous metric $\LC[<\w]$-space $X$ has the $\kappa$-separated $m$-cells property if each non-empty open subset of $X$ contains a separated subset of cardinality $\kappa$.
\end{lemma}

\begin{proof} Assume that each non-empty subset of $X$ contains a separated subset of cardinality $\kappa$.

According to Lemma~\ref{l3}, the $\kappa$-separated $m$-cells property of $X$ will follow as soon as given $\e>0$ we find $\delta>0$ such that for every subset $A\subset X$ of cardinality $|A|<\kappa$ and every map $f:\II^d\to X$ of a cube of finite dimension $d\le m$ there is a map $\tilde f:\II^d\to X$ which is $2\e$-homotopic to $f$ and such that $\dist(\tilde f(\II^d),A)\ge\delta$. 

Being $\II^m{\sim}$homogeneous, the space $X$ contains a point $\theta\in X$ such that the map $$\alpha_\theta:\Iso(X)\to X,\;\;\alpha_\theta:f\mapsto f(\theta),$$is $\II^m{\sim}$invertible. 

Being $\LC[<\w]$, the space $X$ is locally path-connected at $\theta$. Consequently, there is $\delta_1>0$ such that each point $y\in B(\theta,\delta_1)\subset X$ can be linked with $\theta$ by a path of diameter $<\e$. By our hypothesis,  the $\delta_1$-ball $B(\theta,\delta_1)$ contains a separated subset $S\subset B(\theta,\delta_1)$ of size $|S|=\kappa$. Since $S$ is separated, the number $$\delta=\frac13\inf\{\dist(s,t):s,t\in S,\; s\ne t\}$$is strictly positive. 

We claim that the number $\delta$ satisfies our requirements. Indeed, take any subset $A\subset X$ of cardinality $|A|<\kappa$ and fix any map $f:\II^d\to X$ from a cube of finite dimension $d\le m$. Since $X$ is $\LC[<\w]$, there is $\e'>0$ such that any map $f':\II^d\to X$ that is $\e'$-near to $f$ is $\e$-homotopic to $f$ (cf. \cite[V.5.1]{Hu}).

By our hypothesis, the map $\alpha_\theta$ is $\II^m{\sim}$invertible. Therefore
there is a map $g:\II^d\to \Iso(X)$ such that the composition $f'=\alpha_\theta\circ g$ is $\e'$-near to $f$. By the choice of $\e'$, the map $f'$ is $\e$-homotopic to $f$. 

We recall that $$\alpha:\Iso(X)\times X\to X,\;\;\alpha:(f,x)\mapsto f(x),$$denotes the action of the isometry group on $X$.

\begin{claim}\label{claim1} There is a point $s\in S$ such that $\dist(\alpha(g(\II^d)\times\{s\}),A)\ge\delta$. 
\end{claim}

The proof depends on the value of the cardinal $\kappa$. If $\kappa$ is uncountable, then we can fix a dense subset $Q\subset g(\II^d)\times A$ of cardinality $|Q|\le\dens(g(\II^d)\times A)\le\max\{\w,|A|\}<\kappa$.

Assuming that Claim~\ref{claim1} is false,  we could find
for every $s\in S$  a pair $(q_s,a_s)\in Q$ such that $\dist(\alpha(q_s,s),a_s)<\delta$. 

The strict inequality $|Q|<\kappa\le|S|$ implies the existence of two distinct points $s,t\in S$ with $(q_s,a_s)=(q_t,a_t)$. Let $x=q_s=q_t$ and observe that 
$$\begin{aligned}
3\delta\le&\dist(s,t)=\dist(\alpha(x,s),\alpha(x,t))\le\\
\le&\dist(\alpha(q_s,s),a_s)+\dist(a_t,\alpha(q_t,t))<2\delta,\end{aligned}$$ which is a contradiction, proving Claim~\ref{claim1} for an
uncountable $\kappa$.
\smallskip 

In case of a
countable $\kappa$ the argument is a bit different. In this case the set $A$ is finite. We claim that for every $a\in A$ the set  
$$K_a=\{x\in X:\exists y\in g(\II^d)\mbox{ with }\alpha(y,x)=a\}$$ is compact.
This will follow as soon as we check that each sequence $(x_n)_{n\in\w}\subset K_a$ has a cluster point $x_\infty\in K_a$.

For every $n\in\w$ find an isometry $y_n\in g(\II^d)\subset\Iso(X)$ such that $\alpha(y_n,x_n)=a$. By the compactness of $g(\II^d)$, the sequence $(y_n)$ has a cluster point $y_\infty\in g(\II^d)$. Observe that the point $x_\infty=y_\infty^{-1}(a)$ belongs to $K_a$. We claim that $x_\infty$ is a cluster point of the sequence $(x_n)$. 
Given any $\eta>0$ and $n\in\w$, we need to find $p\ge n$ such that $\dist(x_p,x_\infty)<\eta$. 
Since $y_\infty$ is a cluster point of $(y_i)$, there is a number $p\ge n$ such that $\dist(y_\infty(x_\infty),y_p(x_\infty))<\eta$. Then  
$$%\begin{aligned}
\dist(x_p,x_\infty)=\dist(y_p(x_p),y_p(p_\infty))=\dist(a,y_p(x_\infty))=\dist(y_\infty(x_\infty),y_p(x_\infty))<\eta,
%\end{aligned}
$$witnessing that the sets $B_a$, $a\in A$, are compact.

Since the union $K=\bigcup_{a\in A}K_a\subset X$ is compact and the set $S$ is $3\delta$-separated, there is a $s\in S$ such that $\dist(s,K)\ge \delta$. 

We claim that $\dist(\alpha(g(\II^d)\times \{s\}),A)\ge\delta$. Assuming the converse, we would find an isometry $y\in g(\II^d)$ such that $\dist(y(s),a)<\delta$ for some $a\in A$. Let $x=y^{-1}(a)$ and observe that $x\in K_a$ and hence $$\dist(s,K)\le \dist(s,x)=\dist(y(s),y(x))=\dist(y(s),a)<\delta,$$
which contradicts the choice of $s$. This completes the proof of Claim~\ref{claim1}.
\medskip

Define a map $\tilde f:\II^d\to X$ letting $\tilde f(x)=\alpha(g(x),s)$ for $x\in \II^d$. The choice of $s$ ensures that $\dist(\tilde f(\II^d),A)\ge\delta$.

By the choice of $\delta_1$ the point $s\in S\subset B(\theta,\delta_1)$ can be linked with $\theta$ by a path $\gamma:[0,1]\to X$ with $\gamma(0)=\theta$, $\gamma(1)=s$ and $\diam(\gamma[0,1])<\e$. This path allows us to define an $\e$-homotopy 
$$h:\II^d\times[0,1]\to X,\; h:(x,t)\mapsto\alpha(g(x),\gamma(t))$$linking the maps $f'=h_0$ and $\tilde f=h_1$.

Since the map $f'$ is $\e$-homotopic to $f$, we conclude that $\tilde f:\II^d\to X$ is a required map that is $2\e$-homotopic to $f$ and has property $\dist(\tilde f(\II^d),A)\ge\delta$.
\end{proof}

\section{Proof of Theorem~\ref{main}}\label{pf-main}

Let $X$ be an isometrically homogeneous $\II^{<\w}{\sim}$homogeneous metric space. Since each connected component of $X$ is isometrically homogeneous and $\II^{<\w}{\sim}$homoge\-neous, we lose no generality by
assuming that $X$ is connected.
\smallskip

(1) The first statement of Theorem~\ref{main} follows from Theorem~\ref{szenthe}.
\smallskip

(2) Assume that $X$ is a locally Polish ANR-space. We need to prove that $X$ is a manifold modeled on a separable Hilbert space. If the completion $\bar X$ is locally compact, then $X=\bar X$ is an Euclidean manifold according to Theorem~\ref{szenthe}. Therefore
we assume that $\bar X$ is not locally compact.

We claim that the space $X$ has the $\w$-separated $\w$-cells property. This will follow from Lemma~\ref{l4} as soon as we check that each non-empty open subset $U\subset X$ contains an infinite separated subset. Fix any point $x_0\in U$ and find $\e>0$ such that $B(x_0,2\e)\subset U$.

Since the complete metric space $\bar X$ is not locally compact, there is a point $x_1\in\bar X$ having no totally bounded neighborhood. Take any point $x_2\in X$ with $\dist(x_2,x_1)<\e$. 

Since the space $X$ is isometrically homogeneous, there is an isometry $f:X\to X$ such that $f(x_0)=x_2$. This isometry can be extended to an isometry $\bar f:\bar X\to\bar X$. Since the point $x_1$ has no totally bounded neighborhood, the ball $\bar B(x_1,\e)=\{x\in \bar X:\dist(x,x_1)\le\e\}$ contains an infinite separated subset $S\subset X\cap \bar B(x_1,\e)$. Since $\dist(x_1,f(x_0))<\e$, the set $S$ lies in the ball $B(f(x_0),2\e)$. Then $f^{-1}(S)$ is an infinite separated subset of the ball $B(x_0,2\e)\subset U$.

By Lemma~\ref{l2}, the space $X$ has the $\w$-discrete $\w$-cells property and by 
 Toru\'nczyk Theorem~\ref{tor2}, $X$ is an $l_2$-manifold.
\smallskip

(3) Assume that $X$ is a completely-metrizable ANR with LFAP. The isometric homogeneity of $X$ implies that any two connected components of $X$ are isometric.
 Let $\kappa$ be the density of any connected component of $X$. 
The homogeneity of $X$ implies that each non-empty open subset $U\subset X$ has density $\dens(U)\ge\kappa$ (cf. the proof of Corollary 3 in \cite{BZ}). Repeating the proof of Lemma 9 from \cite{BZ} we can also show that for each cardinal $\lambda\le\kappa$ of uncountable cofinality, 
each non-empty open subset $U\subset X$ contains a separated subset $S\subset U$ of cardinality $|S|\ge\lambda$. Applying Lemmata~\ref{l2} and \ref{l4}, 
we conclude that the space $X$ has the $\lambda$-discrete $\w$-cells property for every cardinal $\lambda\le\kappa$ of uncountable cofinality. Since $X$ has LFAP, $X$ has the $\kappa$-discrete $\w$-cells approximation property by Lemma~\ref{l1}. Finally, applying the Toru\'nczyk Characterization Theorem~\ref{tor1}, we conclude that $X$ is an $l_2(\kappa)$-manifold.

\section{Acknowledgements}

This research was supported by the Slovenian Research Agency grants
%BI-UA/09-10/001,
P1-0292-0101, 
J1-9643-0101, and 
J1-2057-0101. We thank the referee for the comments and suggestions.
%\newpage


\begin{thebibliography}{9999}

\bibitem{BZ}
T.~Banakh, I.~Zarichnyy, 
{\em Topological groups and convex sets homeomorphic to non-separable Hilbert spaces}, 
Cent. Eur. J. Math. {\bf 6}:1 (2008) 77--86.

\bibitem{DW} 
D. van Dantzig, B.~L. van der Waerden, {\em 
\"{U}ber metrisch homogene R\"{a}ume},
Abhandlungen Hamburg {\bf 6} (1928) 367--376 . 

\bibitem{DT}
T.~Dobrowolski, H.~Toru\'nczyk, 
{\em Separable complete ANR's admitting a group structure are Hilbert manifolds},
Topology Appl. {\bf 12} (1981), 229--235.

\bibitem{Fer} 
S.~Ferry,
{\em The homeomorphism group of a compact Hilbert cube manifold is an ${\rm ANR}$},
Ann. Math. (2) {\bf 106}:1 (1977) 101--119. 

 \bibitem{Gle} A.~M.~Gleason,
 {\em Groups without small subgroups},
 Ann.\ Math.\ {\bf 56} (1952), 193--212.

\bibitem{HR} 
D.~Halverson, D.~Repov\v s, 
{\em The Bing-Borsuk and the Busemann conjectures}, 
Math. Commun. {\bf 13}:2 (2008), 163--184.

\bibitem{Hofmann} 
K.~H.~Hofmann, 
{\em Homogeneous locally compact groups with compact boundary}, 
Trans. Amer. Math. Soc. {\bf 106} (1963) 52--63.

\bibitem{Hu} 
S.-T.~Hu, 
{\em Theory of Retracts},
Wayne State Univ. Press, Detroit,  1965.

\bibitem{Ke}
A.~Kechris, 
{\em Classical descriptive set theory}, 
Graduate Texts in Mathematics No. 156, 
Springer-Verlag, New York, 1995.

\bibitem{KN} 
S.~Kobayashi, K.~Nomizu,
{\em Foundations of Differential Geometry},
Vol I, Interscience Publ., John Wiley \&\ Sons, New York-London, 1963. 

\bibitem{MS} 
A.~Manoussos, P.~Strantzalos, 
{\em 
On the group of isometries on a locally compact metric space}, 
J. Lie Theory {\bf 13}:1 (2003) 7-12. 

\bibitem{Mi} 
E.~Michael, 
{\em Continuous selections. II,}  
 Ann. Math. (2) {\bf 64} (1956) 562--580.

\bibitem{MZ} 
D.~Montgomery, L.~Zippin,
 {\em Topological Transformation Groups},
 Interscience New York, 1955.

\bibitem{Szenthe} 
J.~Szenthe,
{\em On the topological characterization of transitive Lie group actions}, 
Acta Sci. Math. (Szeged) {\bf 36} (1974) 323--344. 

\bibitem{Tk} 
M.~Tkachenko,
{\em Introduction to topological groups}, 
Topology Appl. {\bf 86}:3 (1998) 179--231.

\bibitem{Tor81} 
H.~Toru\'nczyk, 
{\em Characterizing Hilbert space topology}, 
Fund. Math. 
{\bf 111} (1981), 247--262.

\end{thebibliography}
\end{document}